\documentclass[12pt,a4paper]{amsart}

\usepackage{amsmath}
\usepackage{amssymb}
\usepackage{amsthm}

\newtheorem{theorem}{Theorem}[section]
\newtheorem{proposition}[theorem]{Proposition}
\newtheorem{lemma}[theorem]{Lemma}
\newtheorem{corollary}[theorem]{Corollary}

\theoremstyle{definition}
\newtheorem{definition}[theorem]{Definition}
\theoremstyle{remark}

\newcommand{\Fraisse}{Fra\"{\i}ss\'e}
\newcommand{\AgeK}{\mathcal{K}}
\newcommand{\Lang}{\mathcal{L}}
\newcommand{\Cohen}[1][]{\mathbb{C}_{#1}}
\newcommand{\CantorFiniteSeq}[1][]{\mathbb{C}[#1]}
\newcommand{\Cantor}[1][]{2^{#1}}
\newcommand{\PowersetFin}{\mathcal{P}_{<\omega}}
\newcommand{\PowersetOmega}{\mathcal{P}_{\omega}}
\newcommand{\RO}{\operatorname{RO}}
\newcommand{\Sym}{\operatorname{Sym}}
\newcommand{\forcingisom}{\simeq}
\newcommand{\rest}{\mathord{\upharpoonright}}
\newcommand{\st}{\mathrel{:}}
\newcommand{\defn}[1]{\emph{#1}}
\newcommand{\w}{\omega}
\newcommand{\mbb}[1]{\mathbb{#1}}
\newcommand{\bP}{\mathbb{P}}
\newcommand{\Generic}{G}
\newcommand{\cA}{\mathcal{A}}
\newcommand{\cB}{\mathcal{B}}
\newcommand{\cE}{\mathcal{E}}
\newcommand{\cP}{\mathcal{P}}
\newcommand{\cQ}{\mathcal{Q}}
\newcommand{\cR}{\mathcal{R}}

\newcommand{\cref}[1]{\ref{#1}}

\title{Cohen Forcing with  Fra\"iss\'e Classes Beyond $\aleph_1$}
\author{Mohammad Golshani}
\address{Mohammad Golshani, School of Mathematics, Institute for Research in Fundamental Sciences (IPM), P.O.\ Box:
		19395--5746, Tehran, Iran.}
	\email{golshani.m@gmail.com}

\thanks{The  author's research has been supported by a grant from IPM (No. 1405030417). He thanks Mostafa Mirabi and Nate Ackerman for reading the early draft of the paper and useful comments.}


\begin{document}
\maketitle

\begin{abstract}
Kostana \cite{Kostana} asked for natural conditions on a class $\AgeK$ to ensure that forcing with finite members of $\AgeK$ on an underlying set $X$ is forcing equivalent to Cohen forcing.
We show that
if $\AgeK_{\triangle}$ is the class of finite triangle-free graphs and $|X|\geq\aleph_2$, then
$
\Cohen[X](\AgeK_{\triangle})\not\forcingisom\CantorFiniteSeq[X].
$
We then use a characterization of Cohen algebras due to Milovich to obtain a positive  result. We show that if $\aleph_n\leq|X|$ and $\AgeK$ has $(n+1)$-disjoint amalgamation, then $\Cohen[X](\AgeK)$ is Cohen.
\end{abstract}


\maketitle

\section{Introduction}
Forcing with a Fra\"iss\'e class was introduced in a short note \cite{gol}. It was then studied extensively in \cite{Kostana}, and later generalized in \cite{golmir} to allow function symbols in the language.
Let $\AgeK$ be a class of finite structures in a relational language and let $X$ be
an infinite set.  Let
\[
 \Cohen[X](\AgeK)
  =\{A\in\AgeK: A\subseteq X\},
\]
ordered by reverse inclusion.  When $\AgeK$ is a Fra\"iss\'e class, the generic
filter produces a structure on $X$ whose finite approximations lie in $\AgeK$.
Kostana  asked whether there are natural conditions on $\AgeK$ which ensure that this forcing is equivalent to the Cohen forcing \cite[Question~1]{Kostana}.

A  partial answer was obtained in \cite{AFGMP}, where it was shown that if $\AgeK$ is a non-trivial strong Fra\"iss\'e class, then
$
\Cohen[\w_1](\AgeK)\forcingisom\CantorFiniteSeq[\w_1],
$
for every such $\AgeK$ with at most $\aleph_1$ isomorphism types.
The first purpose of this note is to show that the bound $\aleph_1$ is exact.

\begin{theorem}
\label{Triangle-free forcing is not Cohen}
Suppose $X$ is a set with $|X|\geq\w_2$, and let $\AgeK_{\triangle}$ denote the class of finite triangle-free graphs.
Then
\[
\Cohen[X](\AgeK_{\triangle})\not\forcingisom\CantorFiniteSeq[X].
\]
In particular, $\Cohen[X](\AgeK_{\triangle})$ is not Cohen.
\end{theorem}
 Motivated  by the notion of $n$-disjoint
amalgamation from \cite[Definition~8.5]{AFGMP}, and using Milovich's  characterization of Cohen algebras, see Theorem \ref{Milovich characterization of Cohen algebras},
 we prove the following positive result.
\begin{theorem}\label{thm:higher-dap}
Let $\AgeK$ be a non-trivial strong Fra\"iss\'e class in a finite relational language,
with countably many isomorphism types, and let $X$ be infinite.  Assume that the class $\AgeK$ has $(n+1)$-disjoint amalgamation,  for every
finite $n\geq1$ with
$
   \aleph_n\leq |X|,
$
 Then
$
\Cohen[X](\AgeK)\forcingisom\CantorFiniteSeq[X].
$
\end{theorem}

\section{Preliminaries}
\label{Section: Background for Cohen counterexample}

We use the notation of \cite{AFGMP}.
For an infinite set $X$, let Cohen forcing over $X$ be defined as
\[
\CantorFiniteSeq[X]
=
\bigl\{p\st (\exists X_0\in\PowersetFin(X))\ p\in\Cantor[X_0]\bigr\}.
\]
The order is reverse inclusion.
A partial order $\bP$ is called \defn{Cohen} when $
\RO(\bP)\cong\RO(\CantorFiniteSeq[X])
$, for some infinite set $X$,
We will use the following characterization of Cohen algebras.

\begin{theorem}[\hbox{\cite[Theorem~30.10]{Jech}}]
\label{Regular subalgebra definition of being Cohen}
Suppose $\mbb{B}$ is an infinite Boolean algebra of uniform density.
Then $\mbb{B}$ is Cohen if and only if the set
\[
\{\mbb{A}\subseteq\mbb{B}\st |A|\leq\w\text{ and }\mbb{A}\text{ is a regular subalgebra of }\mbb{B}\}
\]
has a club subset $C$ such that whenever $\mbb{A}_0,\mbb{A}_1\in C$, if $\mbb{A}_2$ is the Boolean algebra generated by $A_0\cup A_1$, then $\mbb{A}_2\in C$.
Further, if the uniform density of $\mbb{B}$ is $\kappa$, then $\mbb{B}$ is forcing isomorphic to $\CantorFiniteSeq[\kappa]$.
\end{theorem}

Let $\AgeK$ be a strong \Fraisse\ class.
The following two results will be used later.

\begin{proposition}[\hbox{\cite[Proposition~4.6]{AFGMP}}]
\label{All Fraisse classes give the same universe}
Let $\AgeK_0$ and $\AgeK_1$ be non-trivial strong \Fraisse\ classes with only countably many isomorphism types.
If $\Generic_0$ is generic for $\Cohen[\w](\AgeK_0)$, then $V[\Generic_0]$ contains a filter which is generic for $\Cohen[\w](\AgeK_1)$.
\end{proposition}

\begin{lemma}[\hbox{\cite[Lemma~4.7]{AFGMP}}]
\label{Uniform density of Cohen structure forcing}
Suppose $X$ is a set and $\AgeK$ is a strong \Fraisse\ class with at most $|X|$-many isomorphism classes.
Then $\RO(\Cohen[X](\AgeK))$ has uniform density $|X|$.
\end{lemma}

\section{Cohen Forcing with Strong \Fraisse\ Classes Beyond $\w_1$: a negative answer}
\label{Section: Cohen beyond omega1}

Let $\AgeK_{\triangle}$ be the class of finite undirected triangle-free graphs. It is a
 non-trivial strong \Fraisse\ class.
Our  aim in this section is to prove Theorem \ref{Triangle-free forcing is not Cohen}.
Let
\[
\iota: \Cohen[X](\AgeK_{\triangle})\longrightarrow
\RO(\Cohen[X](\AgeK_{\triangle}))\setminus\{0\}
\]
be the canonical dense embedding, and let $\mbb{B}$ be the Boolean subalgebra of
$\RO(\Cohen[X](\AgeK_{\triangle}))$ generated by the range of $\iota$, so that
$
\mbb{B}\forcingisom\Cohen[X](\AgeK_{\triangle}).
$
For $Y\subseteq X$, let $\mbb{B}_Y$ be the Boolean subalgebra of $\mbb{B}$ generated by
$
\iota^{\prime\prime}[\Cohen[Y](\AgeK_{\triangle})].
$
We begin with some basic facts about the algebras $\mbb{B}_Y$.

\begin{lemma}
\label{Local triangle algebras detect their supports}
Suppose $X$ is uncountable, $Y,Z \subseteq X$ are countably infinite and
$\mbb{B}_Y\subseteq\mbb{B}_Z$.
Then $Y\subseteq Z$.
\end{lemma}
\begin{proof}
Suppose not, and fix $a\in Y\setminus Z$.
Choose $c\in Y\setminus\{a\}$.
Let $G_{a,c}\in\AgeK_{\triangle}$ be the graph with underlying set $\{a,c\}$ and the single edge $\{a,c\}$, and let
$
G^*_{a,c}=\iota(G_{a,c}).
$
Then $G^*_{a,c}\in\mbb{B}_Y$.
Every permutation $g\in\Sym(X)$ induces an automorphism of
$\Cohen[X](\AgeK_{\triangle})$, and hence an automorphism of $\mbb{B}$, which we also denote by $g$.
Every element of $\mbb{B}_Z$ is fixed pointwise by every permutation which fixes $Z$ pointwise.
Choose $a'\in X\setminus(Z\cup\{a,c\})$ and let $g$ fix $Z\cup\{c\}$ and sends $a$ to $a'$.
Then
$
g(G^*_{a,c})=G^*_{a',c}\neq G^*_{a,c}.
$
But   $G^*_{a,c} \in \mbb{B}_Y\subseteq\mbb{B}_Z$, and
 $g$ is identity on $\mbb{B}_Z$, a contradiction.
\end{proof}
The next lemma can be proved easily.
\begin{lemma}
\label{Local triangle algebras form a club}
The collection
\[
\cE=\{\mbb{B}_Y\st Y\in [X]^\omega\}
\]
contains a club subset of $[\mbb{B}]^{\leq\w}$, where
\end{lemma}

\begin{lemma}
\label{Two incomparable countable supports}
Suppose $|X|\geq\w_2$ and $D\subseteq\PowersetOmega(X)$ is a club.
Then there are $Y_0,Y_1\in D$ such that
\[
|Y_0\cap Y_1|=\w,
\qquad
Y_0\setminus Y_1\neq\emptyset,
\qquad
Y_1\setminus Y_0\neq\emptyset.
\]
\end{lemma}
\begin{proof}
Choose an infinite $R\in D$.
Fix $T\subseteq X\setminus R$ of size $\w_1$.
For each $t\in T$, choose $Y_t\in D$ such that
$R\cup\{t\}\subseteq Y_t$.
Then
$
\left|\bigcup_{t\in T}Y_t\right|\leq\aleph_1<|X|.
$
Choose
$
b\in X\setminus\bigcup_{t\in T}Y_t,
$
and then choose $Y_b\in D$ such that $R\cup\{b\}\subseteq Y_b$.
As $Y_b$ is countable and $T$ is uncountable, there is an $a\in T\setminus Y_b$.
Let
$
Y_0=Y_a
$ and
$
Y_1=Y_b.
$
Then $R\subseteq Y_0\cap Y_1$, $a\in Y_0\setminus Y_1$, and
$b\in Y_1\setminus Y_0$.
\end{proof}

\begin{lemma}
\label{Generated triangle algebra is not regular}
Suppose $Y_0,Y_1\in [X]^\omega$ have infinite intersection, and let
$
a\in Y_0\setminus Y_1,
$
and
$
b\in Y_1\setminus Y_0.
$
Let $\mbb{A}$ be the Boolean subalgebra of $\mbb{B}$ generated by
$B_{Y_0}\cup B_{Y_1}$.
Then $\mbb{A}$ is not a regular subalgebra of $\mbb{B}$.
\end{lemma}
\begin{proof}
Let $G_{a,b}$ be the two element graph with the single edge $\{a,b\}$ and put
$
G^*_{a,b}=\iota(G_{a,b}).
$
We claim that
\[
D=\{d\in \mbb{A}\setminus\{0\}\st d\wedge G^*_{a,b}=0\}
\]
is dense in $\mbb{A}\setminus\{0\}$.
Fix $d\in\mbb{A}\setminus\{0\}$ and suppose that
$d\wedge G^*_{a,b}\neq0$.
Find finite sets $F_i\subseteq Y_i$, for $i\in[2]$, such that every generator occurring in the expression for $d$ belongs to $\mbb{B}_{F_0}$ or $\mbb{B}_{F_1}$.
We may assume that $a\in F_0$ and $b\in F_1$.

As $\iota^{\prime\prime}[\Cohen[X](\AgeK_{\triangle})]$ is dense in $\mbb{B}$, there is a condition
$\cQ\in\Cohen[X](\AgeK_{\triangle})$ such that
$
\iota(\cQ)\leq d\wedge G^*_{a,b}.
$
By strengthening $\cQ$ if necessary, we may also assume
$F_0\cup F_1\subseteq Q$.
For $i\in[2]$, let
$
\cQ_i=\cQ\rest[F_i],
$
and set
$
u=\iota(\cQ_0)\wedge\iota(\cQ_1).
$
By the choice of the sets $F_0,F_1$ and the fact that $\iota(\cQ)\leq d$, we have
$0<u\leq d$.
Choose
\[
c\in(Y_0\cap Y_1)\setminus(F_0\cup F_1).
\]
Let $G_{a,c}$ and $G_{b,c}$ be the corresponding two element one-edge graphs, and put
\[
d^*=u\wedge\iota(G_{a,c})\wedge\iota(G_{b,c}).
\]
Since $\iota(G_{a,c})\in\mbb{B}_{Y_0}$ and
$\iota(G_{b,c})\in\mbb{B}_{Y_1}$, we have $d^*\in\mbb{A}$.

We claim that $d^*\neq0$.
To see this, take the union of $\cQ_0$ and $\cQ_1$, add the edges
$\{a,c\}$ and $\{b,c\}$, and add no other edges between
$F_0\setminus F_1$ and $F_1\setminus F_0$ or incident with $c$.
In particular, do not add the edge $\{a,b\}$.
The resulting finite graph is triangle-free, and gives a condition below $d^*$, so $d^*\neq0$.

On the other hand, every condition below $d^*$ forces both $\{a,c\}$ and $\{b,c\}$ to be edges.
Since all conditions are triangle-free, such a condition cannot contain the edge $\{a,b\}$.
Hence
$
d^*\wedge e_{a,b}=0.
$
This proves that $D$ is dense in $\mbb{A}\setminus\{0\}$.

Let $A_0\subseteq D$ be a maximal antichain in $\mbb{A}$.
Then $A_0$ is maximal in $\mbb{A}$, but every member of $A_0$ is incompatible with the nonzero element $G_{a,b}\in\mbb{B}$.
Therefore $A_0$ is not maximal in $\mbb{B}$, so $\mbb{A}$ is not a regular subalgebra of $\mbb{B}$.
\end{proof}
We are now redy to give the proof of Theorem \cref{Triangle-free forcing is not Cohen}.
Suppose towards a contradiction that
\[
\Cohen[X](\AgeK_{\triangle})\forcingisom\CantorFiniteSeq[X].
\]
Then $\mbb{B}$ is Cohen.
Moreover, by \cref{Uniform density of Cohen structure forcing},
$\mbb{B}$ has uniform density $|X|$.
Hence, by \cref{Regular subalgebra definition of being Cohen}, there is a club
$C\subseteq[\mbb{B}]^{\leq\w}$ such that every member of $C$ is a regular subalgebra of $\mbb{B}$ and, whenever $\mbb{A}_0,\mbb{A}_1\in C$, the Boolean subalgebra generated by
$A_0\cup A_1$ also belongs to $C$.

By \cref{Local triangle algebras form a club}, after intersecting clubs we obtain a club
$D\subseteq\PowersetOmega(X)$ such that
\[
Y\in D\quad\Longrightarrow\quad\mbb{B}_Y\in C.
\]
By \cref{Two incomparable countable supports}, choose $Y_0,Y_1\in D$ with infinite intersection and with
\[
a\in Y_0\setminus Y_1,
\qquad
b\in Y_1\setminus Y_0.
\]
Since $\mbb{B}_{Y_0},\mbb{B}_{Y_1}\in C$, the Boolean subalgebra $\mbb{A}$ generated by
$B_{Y_0}\cup B_{Y_1}$ belongs to $C$ and hence is regular in $\mbb{B}$.
This contradicts the above lemma.

\section{Cohen Forcing with Strong \Fraisse\ Classes Beyond $\w_1$: a positive answer}
\label{Section: Cohen beyond omega1-2}
In this section we prove Theorem \ref{thm:higher-dap}. We start with the definition of disjoint $n$-amalgamation property.
\begin{definition} (see \cite[Definition~8.5]{AFGMP})
\label{Higher DAP definition in Cohen section}
Let $n\geq1$.
A \Fraisse\ class $\AgeK$ has the \defn{disjoint $n$-amalgamation property}
(abbreviated $n$-DAP) if whenever
$\langle \cA_i\rangle _{i\in[n]}\subseteq\AgeK$ are such that
\[
\cA_i\rest[A_i\cap A_j]=\cA_j\rest[A_i\cap A_j]
\]
for all $i<j\in[n]$, there is some $\cB\in\AgeK$ such that
$\cA_i\subseteq\cB$ for all $i\in[n]$.
\end{definition}

Thus $2$-DAP is the same as the strong amalgamation property.
Notice that $\AgeK_{\triangle}$ has $2$-DAP but does not have $3$-DAP.
Indeed, if $a,b,c$ are distinct, the one-edge graphs on
$\{a,b\}$, $\{a,c\}$, and $\{b,c\}$ agree on all pairwise intersections, but they have no common extension in $\AgeK_{\triangle}$.


\begin{theorem}[\hbox{\cite{Milovich}}]
\label{Milovich characterization of Cohen algebras}
Suppose $\mbb{B}$ is a $\pi$-homogeneous Boolean algebra.
Then $\mbb{B}$ is Cohen if and only if there is a club
$C\subseteq[\mbb{B}]^{\leq\w}$ such that, whenever
$1\leq\tau<\w$, $\w_\tau\leq\pi(\mbb{B})$, and
$\mbb{A}_0,\dots,\mbb{A}_{\tau-1}\in C$, the Boolean subalgebra of $\mbb{B}$ generated by
$
\bigcup_{i\in[\tau]}A_i
$
is a regular subalgebra of $\mbb{B}$.
\end{theorem}

Let $\AgeK$ be a non-trivial strong \Fraisse\ class in a finite relational language $\Lang$, and let $X$ be an infinite set.
As above, let $\mbb{B}$ be the Boolean subalgebra of $\RO(\Cohen[X](\AgeK))$ generated by
$\Cohen[X](\AgeK)$, and for $Y\subseteq X$ let $\mbb{B}_Y$ be the subalgebra generated by
$\Cohen[Y](\AgeK)$.

\begin{lemma}
\label{Higher DAP gives regular generated algebras}
Suppose $1\leq\tau<\w$ and $\AgeK$ has $(\tau+1)$-DAP.
Let $Y_0,\dots,Y_{\tau-1}\subseteq X$, and let $\mbb{A}$ be the Boolean subalgebra of
$\mbb{B}$ generated by
$
\bigcup_{i\in[\tau]}B_{Y_i}.
$
Then $\mbb{A}$ is a regular subalgebra of $\mbb{B}$.
\end{lemma}
\begin{proof}
It is enough to show that every element in the dense image of
$\Cohen[X](\AgeK)$ has a reduction to $\mbb{A}$.
Fix $\cP\in\Cohen[X](\AgeK)$.
For $i\in[\tau]$, let
$
\cP_i=\cP\rest[P\cap Y_i],
$
and put
$
r=\bigwedge_{i\in[\tau]}\iota(\cP_i).
$
Then $r\in\mbb{A}$ and $\iota(\cP)\leq r$.
We claim that $r$ is a reduction of $\iota(\cP)$ to $\mbb{A}$.

Fix $0<d\leq r$ with $d\in\mbb{A}$.
Choose finite sets $F_i\subseteq Y_i$ so that membership in $d$ is determined by finitely many generators whose underlying sets are contained in the $F_i$'s.
We may assume $P\cap Y_i\subseteq F_i$ whenever this set is nonempty.
Choose a condition $\cQ\in\Cohen[X](\AgeK)$ with
$
\iota(\cQ)\leq d,
$
and strengthen it so that $\bigcup_{i\in[\tau]}F_i\subseteq Q$.
Let
$
\cQ_i=\cQ\rest[F_i].
$
As in the proof of \cref{Generated triangle algebra is not regular}, the element
$
u=\bigwedge_{i\in[\tau]}\iota(\cQ_i)
$
is nonzero and satisfies $u\leq d$.
The structures
\[
\cP,\cQ_0,\dots,\cQ_{\tau-1}
\]
are pairwise coherent on their intersections.
The structures $\cQ_i$ are pairwise coherent because they are restrictions of the same structure $\cQ$.
Also, since $d\leq r$, the structure $\cQ_i$ agrees with $\cP$ on
$P\cap F_i$.
By $(\tau+1)$-DAP there is a $\cR\in\AgeK$ containing
$\cP$ and every $\cQ_i$.
Hence
\[
0<\iota(\cR)\leq \iota(\cP)\wedge u\leq\iota(\cP)\wedge d.
\]
Thus every nonzero $d\leq r$ in $\mbb{A}$ is compatible with $\iota(\cP)$.
Therefore $r$ is a reduction, and $\mbb{A}$ is regular in $\mbb{B}$.
\end{proof}
The proof of the following lemma is easy.
\begin{lemma}
\label{Elementary submodels give the support club}
Suppose $\theta$ is sufficiently large and regular, and let
$M\prec H(\theta)$ be countable with
$
X,\Lang,\AgeK,\mbb{B}\in M.
$
Let $Y=M\cap X$.
Then
$
\mbb{B}\cap M=\mbb{B}_Y.
$
Consequently, the collection of algebras $\mbb{B}_Y$ contains a club in
$[\mbb{B}]^{\leq\w}$.
\end{lemma}
\begin{theorem}
\label{Higher DAP criterion for Cohen forcing}
Let $\AgeK$ be a non-trivial strong \Fraisse\ class in a finite relational language $\Lang$, and let $X$ be an infinite set.
Suppose that whenever $1\leq\tau<\w$ and
$
\w_\tau\leq|X|,
$
the class $\AgeK$ has $(\tau+1)$-DAP.
Then
$
\Cohen[X](\AgeK)\forcingisom\CantorFiniteSeq[X].
$
\end{theorem}
\begin{proof}
Let $\mbb{B}$ be as above.
By \cref{Uniform density of Cohen structure forcing}, the completion of the forcing has uniform density $|X|$.
Since $\mbb{B}$ is dense in this completion, $\mbb{B}$ is $\pi$-homogeneous and
$
\pi(\mbb{B})=|X|.
$
By \cref{Elementary submodels give the support club}, there is a club
$
C\subseteq[\mbb{B}]^{\leq\w}
$
whose elements are of the form $\mbb{B}_Y$ for countable $Y\subseteq X$.

Fix $1\leq\tau<\w$ with $\w_\tau\leq|X|$, and take
$\mbb{B}_{Y_0},\dots,\mbb{B}_{Y_{\tau-1}}\in C$.
By hypothesis $\AgeK$ has $(\tau+1)$-DAP, so
\cref{Higher DAP gives regular generated algebras} shows that the Boolean subalgebra generated by
$
\bigcup_{i\in[\tau]}B_{Y_i}
$
is regular in $\mbb{B}$.
Thus $C$ witnesses the condition in
\cref{Milovich characterization of Cohen algebras}.
It follows that $\mbb{B}$ is Cohen.
Since
$\mbb{B}\forcingisom\Cohen[X](\AgeK)$ and $\pi(\mbb{B})=|X|$, we obtain
$
\Cohen[X](\AgeK)\forcingisom\CantorFiniteSeq[X].
$
\end{proof}

\begin{corollary}
\label{All finite DAP gives Cohen on every set}
Suppose $\AgeK$ is as in \cref{Higher DAP criterion for Cohen forcing} and has
$n$-DAP for every $2\leq n<\w$.
Then, for every infinite set $X$,
$
\Cohen[X](\AgeK)\forcingisom\CantorFiniteSeq[X].
$
\end{corollary}

\end{document}